\documentclass[a4paper,12pt,reqno]{amsart}

\usepackage{amsfonts}
\usepackage{amsmath}
\usepackage{amssymb}

\usepackage{mathrsfs}
\usepackage{hyperref}
\usepackage[dvipsnames]{xcolor}
\numberwithin{equation}{section}
\theoremstyle{plain}
\newtheorem{thm}{Theorem}[section]
\newtheorem{prop}[thm]{Proposition}
\newtheorem{cor}[thm]{Corollary}
\newtheorem{lemma}[thm]{Lemma}

\theoremstyle{definition}

\def\Rn{{\mathbb R}^n}
\def\X{\mathbb X}
\newcommand{\RomanNumeralCaps}[1]
    {\MakeUppercase{\romannumeral #1}}

\begin{document}
\title[Hardy inequalities on metric measure spaces, \RomanNumeralCaps{2}: 
The case $p>q$]
{Hardy inequalities on metric measure spaces, \RomanNumeralCaps{2}:
 The case $p>q$}

\author[Michael Ruzhansky]{Michael Ruzhansky}
\address{
  Michael Ruzhansky:
    \endgraf
    Department of Mathematics: Analysis, Logic and Discrete Mathematics 
        \endgraf
    Ghent University
      \endgraf
    Krijgslaan 281, Building S8 
      \endgraf
B 9000 Ghent, Belgium
      \endgraf
  and 
  \endgraf
    School of Mathematical Sciences
    \endgraf
  Queen Mary University of London
  \endgraf
  Mile End Road, London E1 4NS
    \endgraf
United Kingdom
  \endgraf
  {\it E-mail address} {\rm m.ruzhansky@qmul.ac.uk}
  }

\author[Daulti Verma]{Daulti Verma}
\address{
  Daulti Verma:
  \endgraf
    Miranda House College
  \endgraf
  University of Delhi
  \endgraf
  Delhi 110007
  \endgraf
  India
  \endgraf
   and
  \endgraf
    School of Mathematical Sciences
  \endgraf
  Queen Mary University of London
  \endgraf
  Mile End Road, London E1 4NS
  \endgraf
  United Kingdom
   \endgraf
  {\it E-mail address} {\rm daulti.verma@mirandahouse.ac.in}
 }
\date{\today}


\thanks{The first author was supported in parts 
by the FWO Odysseus 1 grant G.0H94.18N: Analysis and Partial Differential Equations, by the Methusalem programme of the Ghent University Special Research Fund (BOF) (Grant number 01M01021), and by the EPSRC grant
EP/R003025/2.}

\begin{abstract}
In this paper we continue our investigations giving the characterisation of weights for two-weight Hardy inequalities to hold on general metric measure spaces possessing polar decompositions. Since there may be no differentiable structure on such spaces, the inequalities are given in the integral form in the spirit of Hardy's original inequality.  This is a continuation of our paper \cite{RV-18} where we treated the case $p\leq q$. Here the remaining range $p>q$ is considered, namely, $0<q<p$, $1<p<\infty.$ We give several examples of the obtained results, finding conditions on the weights for integral Hardy inequalities on homogeneous groups, as well as on hyperbolic spaces and on more general Cartan-Hadamard manifolds. As in the first part of this paper, we do not need to impose doubling conditions on the metric.
\end{abstract}

\maketitle
\tableofcontents
\section{Introduction}

After the Hardy inequality was proved by G. H. Hardy in \cite{Hardy:1920}, a large amount of literature is available on this inequality. The integral inequality of the type
\begin{equation}\label{EQ:Hardy_M}
\left(\int_a^b \left(\int_a^x f(t) dt \right)^q u(x) dx \right)^{1/q} \leq C \left( \int_a^b f^p(x)v(x) dx \right)^{1/p},
\end{equation} 
is well-known with $a$ and $b$ real numbers satisfying $-\infty \leq a<b \leq \infty$, and $p, q$ real parameters satisfying $0<q \leq \infty$, $1 \leq p \leq \infty.$ The problem of characterising the weights $u$ and $v$ in this inequality has been investigated by many authors. There are too many references to give an entire overview here, so we refer to only a few: Hardy, Littlewood and Polya \cite{Hardy:1952}, Mitrinovic, Pecaric and Fink \cite{MPF91}, Heinig and Sinnamon \cite{Heinig}, Kokitishvili \cite{VMK},
 Opic and Kufner \cite{OK-book}, Davies \cite{Davies}, 
Kufner, Persson and Samko \cite{Hardy-weighted-book,KP-Samko},
Edmunds and Evans \cite{DE-book}, Mazya \cite{Mazya-Sobolev-spaces,Mazya-book}, Ghoussoub and Moradifam \cite{GM-book}, Balinsky, Evans and Lewis \cite{BEL-book}, and references therein.
We can also refer to the recent open access book
\cite{RS-book} devoted to Hardy, Rellich and other inequalities in the setting of nilpotent Lie groups.

In our previous paper  \cite{RV-18}, for the case $1< p \leq q < \infty$, we characterised the weights $u$ and $v$ for the Hardy inequalities \eqref{EQ:Hardy_M} to hold on general metric measure spaces with polar decompositions.
In this paper, complementary to \cite{RV-18}, we consider the weight characterisations for the case $$0<q<p, 1<p<\infty.$$    
   
  The setting of these papers is rather general, and we consider {\em polarisable metric measure spaces}. These are
  metric spaces $(\mathbb X,d)$ with a Borel measure $dx$ allowing for the following {\em polar decomposition} at $a\in{\mathbb X}$: we assume that there is a locally integrable function $\lambda \in L^1_{loc}$  such that for all $f\in L^1(\mathbb X)$ we have
   \begin{equation}\label{EQ:polar}
   \int_{\mathbb X}f(x)dx= \int_0^{\infty}\int_{\Sigma_r} f(r,\omega) \lambda(r,\omega) d{\omega_r} dr,
   \end{equation}
   where $(r,\omega)\rightarrow a $ as $r\rightarrow0$.
    Here the sets $\Sigma_r=\{x\in G:d(x,a)=r \}\subset \mathbb X$ are equipped with measures, which we denote by $d\omega_r$. 
    
The polar decomposition \eqref{EQ:polar} is a rather general condition in the sense that we allow the function $\lambda$ to be dependent on the full variable $x=(r,\omega)$. 
In the examples described below, in the presence of the differential structure, the function $\lambda(r,\omega)$ can appear naturally as the Jacobian of the polar change of variables. 
However, since we do not assume that $\X$ must have a differentiable structure, we impose
\eqref{EQ:polar} as a condition on metric and measure.

In our previous paper \cite{RV-18} we gave several important examples of polarisable metric spaces, let us briefly recapture them here:
\begin{itemize}
\item[(I)] On the Euclidean space $\Rn$, we can take $ \lambda (r,\omega)= {r}^{n-1}$, and more generally, we have \eqref{EQ:polar} on all 
homogeneous groups with $ \lambda (r,\omega)= {r}^{Q-1}$, where $Q$ is the homogeneous dimension of the group. We can also refer to Folland and Stein \cite{FS-Hardy} and to \cite{FR} for details of such groups.
\item[(II)] Hyperbolic spaces $\mathbb H^n$ with $\lambda(r,\omega)=(\sinh {r})^{n-1}$, or more general symmetric spaces of noncompact type.
\item[(III)] Cartan-Hadamard manifolds, that is, complete, simply connected Riemannian manifolds with non-positive sectional curvature. In this case $\lambda(\rho,\omega)$ depends on both variables $\rho$ and $\omega$. 
We refer to Section \ref{SEC:CH} for this example, and to \cite{RV-18} for more details on  $\lambda(\rho,\omega)$ in this case.
 \item[(IV)] Arbitrary complete Riemannian manifolds $M$: let $C(p)$ denote the cut locus of a point $p\in M$, which we may fix. 
 Let us denote by $M_{p}$ the tangent space to $M$ at $p$, and by $|\cdot|$ the Riemannian length. 
 We also denote $D_{p}:=M\backslash C(p)$ and $S(p;r):=\{x\in M_{p}:|x|=r\}$. Then for any integrable function $f$ on $M$ we have the polar decomposition
        \begin{equation}\label{pol_decom_manif}
        \int_{M}f dV=\int_{0}^{+\infty}dr\int_{r^{-1}S(p;r)\cap D_{p}}f(\exp r\xi)\sqrt{g}(r;\xi)d\mu_{p}(\xi)
        \end{equation}
        for some function $\sqrt{g}$ on $D_{p}$, where $r^{-1}S(p,r)\cap D_{p}$ is the subset of $S_{p}$ obtained by dividing each of the elements of $S(p,r)\cap D_{p}$ by $r$, and $S_{p}:=S(p;1)$. The measure $d\mu_{p}(\xi)$ is the Riemannian measure on $S_{p}$ induced by the Euclidean Lebesgue measure on $M_{p}$. We refer to \cite[Formula III.3.5, P.123]{Cha06}, \cite[Chapter 4]{Li12} and \cite[Chapter 1, Paragraph 12]{CLN06} for more information on this decomposition.
   \end{itemize} 

In this paper, as usual, we will write $ A\approx B$ to indicate  that the expressions $A$ and $B$ are equivalent.

The authors would like to thank Aidyn Kassymov and Bolys Sabitbek for checking some calculations in this paper.

\section{Main results}

Let $d$ is the metric on $\X$. 
We denote by $B(a,r)$ the corresponding balls with respect to $d$, centred at $a\in\X$ and having radius  $r$, namely, 
$$B(a,r):= \{x\in\mathbb X : d(x,a)<r\}.$$ 
To simplify the notation, for all arguments, we fix some point $a\in {\mathbb X}$, and then we will  denote  $${\vert x \vert}_a := d(a,x).$$

The main result of this paper is to characterise the weights $u$ and $v$ for which the corresponding Hardy inequality holds on $\X$. For $\X$=$\mathbb {R}$, such result has been proved by Sinnamon and Stepanov \cite{SS-96}. 
For an alternative approach to these estimates in the case $1<q<p<\infty$ we refer to \cite[Theorem 1.13]{RY18b}.
Now, we formulate one of our main results:
\begin{thm}\label{Theorem1} 
Suppose $0<q<p$, $1<p<\infty$ and $1/r=1/q-1/p$. Let $\mathbb X $ be a metric measure space with a polar decomposition at $a$. Let $u,v> 0$ be measurable and positive a.e in $\mathbb X$ such that $u\in L^1(\mathbb X\backslash \{a\})$ and $v^{1-p'}\in L^1_{loc}(\mathbb X)$. 

Then the inequality
\begin{equation}\label{EQ:Hardy1}
\bigg(\int_\mathbb X\bigg(\int_{B(a,\vert x \vert_a)}\vert f(y) \vert dy\bigg)^q u(x)dx\bigg)^\frac{1}{q}\le C\bigg\{\int_{\mathbb X} {\vert f(x) \vert}^pv(x)dx\bigg\}^{\frac1p}
\end{equation}

holds for all measurable functions $f:\X\to{\mathbb C}$ if and only if 
$$
\mathcal A_2:=  \bigg(\int_\mathbb X \bigg(\int_{\mathbb X\backslash{B(a,|x|_a )}} u(y) dy\bigg)^{r/p}  \bigg( \int_{B(a,|x|_a  )}v^{1-p'}(y)dy\bigg)^{r/p'}u(x)dx\bigg)^{1/r}<\infty.
$$
Moreover, the smallest constant $C$ for which \eqref{EQ:Hardy1} holds satisfies 
$$
(p')^{1/p'} q^{1/p} (1-q/p) \mathcal A_2 \le C \le (r/q)^{1/r} p^{1/p} p'^{1/p'} \mathcal A_2.
$$
\end{thm} 

Before proving the above theorem, we will need to prove several auxiliary facts.
Throughout this paper, we will use the following notations:
\begin{align}\label{EQ:U}
U(x)= { \int_{\mathbb X\backslash{B(a,|x|_a )}} u(y) dy}, 
\end{align} 
\begin{align} \label{EQ:V}
V(x)= \int_{B(a,|x|_a  )}v^{1-p'}(y)dy,
\end{align}
 \begin{align}\label{EQ:U_tilde}
  \tilde{U}(t)= \int_t^\infty\int_{\Sigma_\rho} \lambda(\rho,\sigma)u(\rho,\sigma) d\sigma_{\rho} d\rho,
 \end{align} 
 \begin{align}\label{EQ:V_tilde}
   \tilde{V}(t)= \int_0^t \int_{\Sigma_\rho} \lambda(\rho,\sigma)v^{1-p'}(\rho,\sigma) d\sigma_{\rho} d\rho,
 \end{align} 
 \begin{align}\label{EQ:U_1}
 U_1(\rho)=\int_{\Sigma_\rho} \lambda(\rho,\sigma)u(\rho,\sigma) d\sigma_{\rho},
 \end{align} 
 \begin{align}\label{EQ:V_1}
  V_1(\rho)=\int_{\Sigma_\rho} \lambda(\rho,\sigma)v^{1-p'}(\rho,\sigma) d\sigma_\rho.
 \end{align} 
 \begin{lemma}
 \label{lemma1}
 Let us denote
$$
\mathcal A_{1} := \bigg\{\int_{\mathbb X} U^{r/q}(x) V^{r/q'}(x)v^{1-p'}(x)dx \bigg\}^{1/r}.
$$
Then
 \begin{align}\label{EQ:Equi1}
 \mathcal A_2^r =(q/p')  \mathcal A_1^r.
  \end{align}
  \end{lemma}
  \begin{proof} Using integration by parts, we have
   \begin{align}
&  \mathcal A_2^r =\int_{\mathbb X} U^{r/p}(x)V^{r/p'}(x)u(x)dx \nonumber
\\
&= \int_0^\infty \tilde{U}^{r/p}(t) \tilde{V}^{r/p'}(t) U_1(t)dt\nonumber
\\
&=\int_0^\infty \bigg(\int_t^\infty U_1(\rho)d\rho \bigg)^{r/p} \tilde{V}^{r/p'}(t) U_1(t)dt\nonumber
\\
&=-(q/r)U^{r/q}(\infty)V^{r/p'}(\infty)+(q/r)U^{r/q}(0)V^{r/p'}(0)+({q/r})(r/p')\nonumber
\\
&\quad\times \int_{\mathbb X} U^{r/q}(x)V^{r/q'}(x)v^{1-p'}(x)dx\nonumber
\\
& = ({q/p'}) \int_{\mathbb X} U^{r/q}(x)V^{r/q'}(x)v^{1-p'}(x)dx \nonumber
\\
&=({q/p'}) \mathcal A_1 ^r, \nonumber 
  \end{align}
  completing the proof.
\end{proof}

\begin{lemma}
\label{lemma2}
Suppose that $\alpha, \beta$ and $\gamma$ are non-negative functions and $\gamma $ is a radial non-decreasing function of $\vert \cdot \vert_a$.
 If $\int_{\mathbb X \backslash{B(a,|x|_a )}} \alpha(y)dy\le \int_{\mathbb X \backslash{B(a,|x|_a )}} \beta(y)dy $ for all $x$, then $\int_ \mathbb X\gamma \alpha \le \int_\mathbb X\gamma \beta$.
 \end{lemma}
 \begin{proof}
  Let us denote
 $$\alpha_1(r)=\int_{\Sigma_r} \lambda(r,\sigma)\alpha(r,\sigma) d\sigma_r,$$
$$\beta_1(r)=\int_{\Sigma_r} \lambda(r,\sigma) \beta(r,\sigma) d\sigma_r,$$
$$\tilde{\gamma}(r)=\gamma(x),$$ \rm for $ \vert x \vert_a=r.$
 Given that,
$ \int_{\mathbb X \backslash{B(a,|x|_a )}} \alpha(y)dy\le \int_{\mathbb X \backslash{B(a,|x|_a )}} \beta(y)dy$,
changing to polar coordinates, we get
$$ \int_{\vert x \vert_a}^\infty {\alpha_1}(r)dr \le  \int_{\vert x \vert_a}^\infty {\beta_1}(r)dr.$$ 
\\
Using \cite[Lemma 2.1]{SS-96} which says if $\alpha, \beta$, $\gamma $ are non-negative functions and $\gamma $ is non-decreasing, and if $\int_x^\infty \alpha(y)dy\le \int_x^\infty \beta(y)dy $ for all $x$, then $\int_0^\infty\gamma\alpha\le\int_0^\infty\gamma\beta$.
Therefore,
\begin{align}
&&
\int_\mathbb X \gamma(x) \alpha(x)dx= \int_0^\infty \tilde{\gamma}(r) {\alpha_1}(r)dr\nonumber
& \le  \int_0^\infty \tilde{\gamma}(r) {\beta_1}(r)dr\nonumber=\int_\mathbb X \gamma(x) \beta(x)dx,
\end{align}
completing the proof.
\end{proof} 

\begin{prop}
\label{proposition1}
 Suppose that $u,b$ and $F$ are non-negative functions with $F$ non-decreasing such that $\int_{\mathbb X \backslash{B(a,|x|_a )}} b(y)dy<\infty$ for all $x\neq a$ and $\int_{\mathbb X} b(x)dx=\infty$.
If $0<q<p<\infty$, $F$ is radial in $\vert x \vert_a$, and $ 1/r=1/q-1/p$, then 
\begin{align}
&\bigg(\int_{\mathbb X}  F^q(x) u(x)dx \bigg)^{1/q}\nonumber
\\
&  \le(r/p)^{1/r} \bigg(\int_{\mathbb X} \bigg(\int_{\mathbb X \backslash{B(a,|x|_a )}} u(y)dy \bigg)^{r/q} \bigg(\int_{\mathbb X \backslash{B(a,|x|_a )}} b(y) dy\bigg)^{-r/q} b(x) dx\bigg)^{1/r}\nonumber 
\\
& \quad\times \bigg(\int_{\mathbb X} F^p(x) b(x)dx \bigg)^{1/p}.\nonumber
\end{align}
\end{prop}
\begin{proof}

Let us denote
\\
\begin{align*}
B(x) &= \int_{\mathbb X \backslash{B(a,|x|_a )}} b(y)dy,\\
  \tilde{B}(t) &=\int_{t}^ {\infty}\int_{\Sigma_\rho} \lambda(\rho,\omega) b(\rho,\omega)d\omega_{\rho} d\rho,\\
  B_1(\rho) &=\int_{\Sigma_\rho} \lambda(\rho,\omega) b(\rho,\omega)d\omega_{\rho},\\
   \tilde{F}(t) &=F(x) , \ \textrm {for} \ \vert x \vert_a=t.
   \end{align*}
Applying H\"older's inequality with indices $q/r$ and $q/p$, we get 

\begin {align}
&
\bigg(\int_{\mathbb X}  F^q(x) u(x)dx \bigg)^{1/q}\nonumber
\\
&=\bigg(\int_{\mathbb X}  \bigg(\int_{B(a,|x|_a )} U^{r/p}(y) B^{-r/q}(y) b(y)dy \bigg)^{q/r} F^q(x)\nonumber
\\
& \quad\times \bigg(\int_{B(a,|x|_a )} U^{r/p}(y) B^{-r/q}(y) b(y)dy \bigg)^{-q/r} u(x)dx \bigg)^{1/q} \nonumber
\\
& =\bigg(\int_0^\infty  \bigg(\int_0^t \tilde {U}^{r/p}(\rho) \tilde {B}^{-r/q}(\rho) B_1(\rho)d\rho \bigg)^{q/r} \tilde{F}^q(t)  \bigg(\int_0^t \tilde {U}^{r/p}(\rho) \tilde {B}^{-r/q}(\rho) B_1(\rho)d\rho \bigg)^{-q/r} U_1(t) dt\bigg)^{1/q}\nonumber
\\
& =\bigg(\int_0^\infty  \bigg(\int_0^t \tilde {U}^{r/p}(\rho) \tilde {B}^{-r/q}(\rho) B_1(\rho)d\rho \bigg)^{q/r} \tilde{F}^q(t)  \bigg(\int_0^t \tilde {U}^{r/p}(\rho) \tilde {B}^{-r/q}(\rho) B_1(\rho)d\rho \bigg)^{-q/r}\nonumber
\\
 &\quad\times U_1^{q/r+q/p}(t) dt\bigg)^{1/q}\nonumber
\\
&\le  \bigg(\int_0^\infty \bigg(\int_0^t \tilde {U}^{r/p}(\rho) \tilde {B}^{-r/q}(\rho) B_1(\rho)d\rho \bigg) U_1(t)dt \bigg)^{1/r} \bigg(\int_0^\infty \tilde {F}^p(t) \bigg(\int_0^t \tilde {U}^{r/p}(\rho) \tilde {B}^{-r/q}(\rho) B_1(\rho)d\rho \bigg)^{-p/r}\nonumber
\\
&\quad\times U_1(t)dt \bigg)^{1/p} .\nonumber 
\end{align} 

On interchanging the order of integration and using $r/p+1=r/q$, the first factor becomes
\begin{align}
 & \bigg(\int_0^\infty  \tilde {U}^{r/p}(\rho) \tilde {B}^{-r/q}(\rho) B_1(\rho) \bigg(\int_\rho^{\infty} U_1(t)dt\bigg) d\rho \bigg)^{1/r}\nonumber
\\
&= \bigg(\int_0^\infty  \tilde {U}^{r/q}(\rho) \tilde {B}^{-r/q}(\rho) B_1(\rho) d\rho \bigg)^{1/r}\nonumber
\\
 &=\bigg(\int_{\mathbb X} \bigg(\int_{\mathbb X \backslash{B(a,|x|_a )}} u(y)dy \bigg)^{r/q} \bigg(\int_{\mathbb X \backslash{B(a,|x|_a )}} b(y) dy\bigg)^{-r/q} b(x) dx\bigg)^{1/r}.\nonumber
 \end{align}
 \\
  To complete the proof we apply Lemma \ref{lemma2} to the second factor. We take $\alpha(x)=\bigg(\int_{B(a,|x|_a )} U^{r/p}(y) B^{-r/q}(y)b(y)dy \bigg)^{-p/r} u(x)$, $\beta(x)=(r/p)^{p/r}b(x)$, and $\gamma(x)=F^p(x)$ in Lemma  \ref{lemma2}. As $\gamma$ is  non-decreasing by assumption, it remains to check that 
  $$\int_{\mathbb X \backslash{B(a,|x|_a )}} \alpha(y)dy \le \int_{\mathbb X \backslash{B(a,|x|_a )}} \beta(y)dy,$$ for all $x$. Since $\bigg(\int_0^t \tilde {U}^{r/p}(\rho) \tilde {B}^{-r/q}(\rho) B_1(\rho)d\rho \bigg)^{-p/r}$ and $\tilde{U}$ are non-increasing,
\begin{align}
\int_{\mathbb X \backslash{B(a,|x|_a )}} \alpha(y)dy\nonumber
&= \int_{\vert x \vert_a}^\infty \bigg(\int_0^t \tilde {U}^{r/p}(\rho) \tilde {B}^{-r/q} (\rho)B_1(\rho)d\rho \bigg)^{-p/r} U_1(t)dt\nonumber
\\
& \le \bigg(\int_0^{ \vert x \vert_a}  \tilde {U}^{r/p}(\rho) \tilde {B}^{-r/q}(\rho) B_1(\rho)d\rho \bigg)^{-p/r} \int_{\vert x \vert_a}^\infty U_1(t)dt\nonumber
\\
& \le \bigg(\int_0^{ \vert x \vert_a}   \tilde {B}^{-r/q}(\rho) B_1(\rho)d\rho \bigg)^{-p/r} \tilde U^{-1}(\vert x \vert_a) \int_{\vert x \vert_a}^\infty U_1(t)dt\nonumber
\\
& =\bigg(\int_0^{ \vert x \vert_a}  \tilde {B}^{-r/q}(\rho) B_1(\rho)d\rho \bigg)^{-p/r}\nonumber
\\
&=\bigg(\int_0^{ \vert x \vert_a}  \tilde {B}^{-r/q}(\rho) d(-\tilde{B}(\rho)) \bigg)^{-p/r}\nonumber
\\
&=\bigg((p/r)  B^{-r/p}(x)\bigg)^{-p/r} =\int_{\mathbb X \backslash{B(a,|x|_a )}} \beta(y)dy. \nonumber
\end{align}
Finally, by using Lemma \ref{lemma2} we get
\begin {align}
&
\bigg(\int_{\mathbb X}  F^q(x) u(x)dx \bigg)^{1/q}\nonumber
\\
&\le  \bigg(\int_{\mathbb X} \bigg(\int_{\mathbb X \backslash{B(a,|x|_a )}} u(y)dy \bigg)^{r/q} \bigg(\int_{\mathbb X \backslash{B(a,|x|_a )}} b(y) dy\bigg)^{-r/q} b(x) dx\bigg)^{1/r}  \bigg(\int_{\mathbb X} (r/p)^{p/r} F^p(x) b(x)dx \bigg)^{1/p}\nonumber
\\
&= (r/p)^{1/r} \bigg(\int_{\mathbb X} \bigg(\int_{\mathbb X \backslash{B(a,|x|_a )}} u(y)dy \bigg)^{r/q} \bigg(\int_{\mathbb X \backslash{B(a,|x|_a )}} b(y) dy\bigg)^{-r/q} b(x) dx\bigg)^{1/r}  \bigg(\int_{\mathbb X} F^p(x) b(x)dx \bigg)^{1/p}\nonumber,
\end{align}
which completes the proof.
\end{proof}

\begin{prop}
\label{proposition2}
Suppose $1<p<\infty$ and $w$ is a non-negative function satisfying
\\
  \begin{align}\label{EQ:Condn}
   0<\displaystyle \int_{B(a,|x|_a )} w(y)dy <\infty , \forall x\neq a, \displaystyle \int_{\mathbb X} w(x)dx =\infty. 
   \end{align}
\\
Then 
\\
\begin{align}\label{EQ:Prepo}
\bigg(\displaystyle \int_{\mathbb X}\bigg( \int_{B(a,|x|_a )} f(y)dy \bigg)^p \bigg( \int_{B(a,|x|_a )} w(y)dy \bigg)^{-p} w(x)dx\bigg)^{1/p}\nonumber
\\
\le p' \bigg(\int_{\mathbb X} f^p(x) w^{1-p}(x)dx\bigg)^{1/p},
\end{align}
for all measurable functions $f\ge0.$
\end{prop}
\begin{proof} Let us denote:
${f_1}(\rho)=\int_{\Sigma_\rho}  \lambda(\rho,\sigma) f(\rho,\sigma) d\sigma_\rho $, 
${w_1}(\rho)=\int_{\Sigma_\rho}  \lambda(\rho,\sigma) w(\rho,\sigma) d\sigma_\rho .$

Consider the left hand side of \eqref{EQ:Prepo} and change it into polar coordinates, to get

\begin{align}
& \bigg(\displaystyle \int_{\mathbb X}\bigg( \int_{B(a,|x|_a )} f(y)dy \bigg)^p \bigg( \int_{B(a,|x|_a )} w(y)dy \bigg)^{-p} w(x)dx\bigg)^{1/p}\nonumber
\\
&=\bigg(\displaystyle \int_0^\infty \bigg( \int_0^t {f_1}(\rho)d\rho \bigg)^p \bigg( \int_0^t {w_1}(\rho)d\rho \bigg)^{-p} {w_1}(t)dt\bigg)^{1/p}\nonumber.
\end{align}
Now, let us use \cite[Proposition 2.3]{SS-96} which says that if $1<p<\infty$ and $w$ is a non-negative function satisfying 
\\
  \begin{align}
   & 0<\displaystyle \int_0^x w(y)dy <\infty , \forall x> 0, \displaystyle \int_0^\infty w(x)dx =\infty,\nonumber 
   \end{align}
   then 
   \begin{align}
\bigg(\displaystyle \int_0^\infty \bigg( \int_0^t f(\rho)d\rho \bigg)^p \bigg( \int_0^t w(\rho)d\rho \bigg)^{-p} w(t)dt\bigg)^{1/p}\nonumber
\\
\le p' \bigg(\int_0^\infty f^p(t) w^{1-p}(t)dt\bigg)^{1/p}.\nonumber
\end{align}

 By using  H\"older's inequality to the indices $1/p$ and $1/p'$, the LHS of \eqref{EQ:Prepo} can be estimated by
\begin{align}
&
 \le  p' \bigg(\int_0^\infty {f_1}^p(t) {w_1}^{1-p}(t)dt\bigg)^{1/p}\nonumber 
 \\
 &= p' \bigg(\int_0^\infty \bigg(\int_{\Sigma_t}  \lambda(t,\sigma) f(t,\sigma) d\sigma_t \bigg)^p \bigg(\int_{\Sigma_t}  \lambda(t,\sigma) w(t,\sigma) d\sigma_t \bigg)^{1-p} dt\bigg)^{1/p}\nonumber
 \\
 &= p' \bigg(\int_0^\infty \bigg(\int_{\Sigma_t}  \lambda(t,\sigma) f(t,\sigma) w^{(1-p)/p+(p-1)/p}d\sigma_t \bigg)^p \bigg(\int_{\Sigma_t}  \lambda(t,\sigma) w(t,\sigma) d\sigma_t \bigg)^{1-p} dt \bigg)^{1/p}\nonumber
 \\
&\le p' \bigg(\int_0^\infty \bigg(\int_{\Sigma_t}  \lambda(t,\sigma) f^p(t,\sigma) w^{1-p}d\sigma_t \bigg)\bigg(\int_{\Sigma_t}  \lambda(t,\sigma) w^{p'(p-1)/p}(t,\sigma) d\sigma_t \bigg)^{p-1}\nonumber
\\
& \quad\times \bigg(\int_{\Sigma_t}  \lambda(t,\sigma) w(t,\sigma) d\sigma_t \bigg)^{1-p} dt \bigg)^{1/p}\nonumber
\\
&
= p' \bigg(\int_0^\infty \int_{\Sigma_t}  \lambda(t,\sigma) f^p(t,\sigma) w^{1-p}(t,\sigma) d\sigma_t dt \bigg)^{1/p}\nonumber
\\
&= p' \bigg(\displaystyle\int_{\mathbb X} f^p(x) w^{1-p}(x) dx \bigg)^{1/p},\nonumber
\end{align}
completing the proof.
\end{proof}

Now, we prove our Theorem \ref{Theorem1} :
\begin{proof} 
Set $w=v^{1-p'}$. Suppose that inequality \eqref{EQ:Hardy1}  holds for all $f\ge0$ and let $u_0$ and $w_0$ be $L^1$ functions  such that $0 < u_0 \le u$ and $0 < w_0 \le w$. 
We denote
 $$ \tilde{u_0}(\rho)=\int_{\Sigma_{\rho}} \lambda(\rho,\omega)u_0(\rho, \omega)d\omega_{\rho},$$ 
 $$ \tilde{w_0}(\rho)=\int_{\Sigma_{\rho}} \lambda(\rho,\omega)w_0(\rho, \omega)d\omega_{\rho}.$$

Let us apply inequality \eqref{EQ:Hardy1} to the function
$$f(x)=\displaystyle\bigg( \int_{\mathbb X \backslash{B(a,|x|_a )}}u_0(y)dy \bigg)^{r/pq} \bigg(\int_{B(a,|x|_a )} w_0(y)dy\bigg)^{r/pq'} w_0(x).$$

After changing to polar coordinates and using $$ \frac{r}{(pq')}+1=r(\frac{1}{pq'}+\frac{1}{r})\\
=r\bigg(\frac{1}{p}(1-\frac{1}{q})+\frac{1}{q}-\frac{1}{p}\bigg)=\frac{r}{p'q},$$ 
we have
\begin{align*} 
&
\int_{B(a,|x|_a )} f(y)dy = \int_0^{\vert x \vert_a} \int_{\Sigma_{t}} \lambda(t,\sigma) \bigg(\int_{t}^\infty \int_{\Sigma_{\rho}} \lambda(\rho,\omega)u_0(\rho,\omega)d\rho d\omega_{\rho} \bigg)^{r/pq}\nonumber
\\
& \quad\times \bigg(\int_0^t \int_{\Sigma_{\rho}} \lambda(\rho,\omega) w_0 (\rho, \omega)d\rho d\omega_{\rho} \bigg)^{r/pq'}w_0(t,\sigma)dt d\sigma_t\nonumber 
\\
&=\int_0^{\vert x \vert_a} \bigg(\int_{t}^\infty  \tilde u_0(\rho)d\rho \bigg)^{r/pq} \bigg(\int_0^t \tilde w_0 (\rho)d\rho \bigg)^{r/pq'} \tilde w_0(t)dt\nonumber
\\
& \ge  \bigg(\int_{{\vert x \vert_a}}^\infty  \tilde u_0(\rho)d\rho \bigg)^{r/pq} \int_0^{\vert x \vert_a} \bigg(\int_0^t \tilde w_0 (\rho)d\rho \bigg)^{r/pq'} \tilde w_0(t)dt\nonumber
\\
& =(p'q/r) \bigg(\int_{\vert x \vert_a}^\infty  \tilde u_0(\rho)d\rho \bigg)^{r/pq} \bigg(\int_0^{\vert x \vert_a} \tilde w_0 (\rho)d\rho \bigg)^{r/p'q}\nonumber 
\\
&=(p'q/r) \bigg( \int_{\mathbb X \backslash{B(a,|x|_a )}}u_0(y)dy \bigg)^{r/pq} \bigg(\int_{B(a,|x|_a )} w_0(y)dy\bigg)^{r/p'q}. \nonumber
\end{align*}

Observe that $w=v^{1-p'}$ implies that $v=w^{1/(1-p')}=w^{1-p}$ since $$1-1/(1-p')=-p'/(1-p')=-1/(1/p'-1)=p.$$
We then have
\begin{align}
 &\bigg(\int_{\mathbb X} (p'q/r)^q  \bigg( \int_{\mathbb X \backslash{B(a,|x|_a )}}u_0(y)dy \bigg)^{r/p} \bigg(\int_{B(a,|x|_a )} w_0(y)dy\bigg)^{r/p'} u_0(x)dx \bigg)^{1/q}\nonumber
 \\
 & \le \bigg(\int_\mathbb X\bigg(\int_{B(a,\vert x \vert_a)} f(y)  dy\bigg)^q u(x)dx\bigg)^{1/q}\nonumber
 \\
 & \le C \bigg(\int_{\mathbb X} f^p(x) w^{1-p}(x)dx \bigg)^{1/p}\nonumber
 \\
 &=C \bigg(\int_{\mathbb X}  \bigg( \int_{\mathbb X \backslash{B(a,|x|_a )}}u_0(y)dy \bigg)^{r/q} \bigg(\int_{B(a,|x|_a )} w_0(y)dy\bigg)^{r/q'} w_0^p(x) w^{1-p}(x)dx \bigg)^{1/p}\nonumber
 \\
 & \le C \bigg(\int_{\mathbb X}  \bigg( \int_{\mathbb X \backslash{B(a,|x|_a )}}u_0(y)dy \bigg)^{r/q} \bigg(\int_{B(a,|x|_a )} w_0(y)dy\bigg)^{r/q'} w_0^p(x) w_0^{1-p}(x)dx \bigg)^{1/p}\nonumber
 \\
 &= C  \bigg(\int_{\mathbb X}  \bigg( \int_{\mathbb X \backslash{B(a,|x|_a )}}u_0(y)dy \bigg)^{r/q} \bigg(\int_{B(a,|x|_a )} w_0(y)dy\bigg)^{r/q'} w_0(x)dx \bigg)^{1/p}\nonumber
 \\
 &=C   \bigg(\int_{\mathbb X} (p'/r)(r/q) \bigg( \int_{\mathbb X \backslash{B(a,|x|_a )}}u_0(y)dy \bigg)^{r/p} \bigg(\int_{B(a,|x|_a )} w_0(y)dy\bigg)^{r/p'} u_0(x)dx \bigg)^{1/p}\nonumber
 \\
 &= C (p'/q)^{1/p}  \bigg(\int_{\mathbb X}  \bigg( \int_{\mathbb X \backslash{B(a,|x|_a )}}u_0(y)dy \bigg)^{r/p} \bigg(\int_{B(a,|x|_a )} w_0(y)dy\bigg)^{r/p'} u_0(x)dx \bigg)^{1/p},\nonumber
 \end{align}
where the second last equality is integration by parts. Since   $u_0$ and $w_0$ are in $L^1$ and are positive, the integral on the right hand side is finite. Therefore, we have

$$
 (p'q/r)(q/p')^{1/p}  \bigg(\int_\mathbb X \bigg(\int_{\mathbb X \backslash{B(a,|x|_a )}} u_0(y)dy \bigg)^{r/p}  \bigg( \int_{B(a,|x|_a )}w_0(y)dy \bigg)^{r/p'} u_0(x)dx\bigg)^{1/r} \le C.
$$
Approximating $u$ and $w$ by increasing sequence of $L^1$ functions, using $$ (p'q/r)(q/p')^{1/p} =(p' )(p')^{-(1/p)}q^{1/p}(q/r)=(p')^{1/p'}q^{1/p}(q(1/q-1/p))=(p')^{1/p'} q^{1/p} (1-q/p) $$  and applying the Monotone Convergence Theorem, we conclude that
$$
(p')^{1/p'} q^{1/p} (1-q/p) \mathcal A_2 \le C.
$$

Suppose now that $\mathcal A_2<\infty$ and, for the moment, that \eqref{EQ:Condn} holds for $w$. 
Set $V(x)=\displaystyle \int_{B(a,|x|_a )} w(y)dy$ and apply Proposition \ref{proposition1} with $b=V^{-p}w$ and $F(x)=\int_{B(a,|x|_a )} f(y)dy $. 
Let us denote $$ \tilde{V}(t)= \int_0^t \int_{\Sigma_\rho}\lambda(\rho,\sigma) w(\rho,\sigma) d\rho d\sigma_\rho ,$$ $$ V_1(\rho)=\int_{\Sigma_\rho} \lambda(\rho,\sigma)w(\rho,\sigma) d\sigma_\rho, $$ $$ U_1(\rho)= \int_{\Sigma_\rho} \lambda(\rho,\sigma)u(\rho,\sigma) d\sigma_\rho. $$

Also, 
$$\int_{\mathbb X \backslash{B(a,|x|_a )}} b(y)dy <\infty,$$
since $$\int_{\mathbb X \backslash{B(a,|x|_a )}} b(y)dy=\int_{\mathbb X \backslash{B(a,|x|_a )}} V^{-p}(y)w(y)dy =\int_{\vert x \vert_a}^\infty \tilde V^{-p}(\rho) V_1(\rho)d\rho=(p'/p) \tilde V^{1-p}(\vert x \vert_a) .$$ 

The conclusion of Proposition \ref{proposition1} becomes

 \begin{align}
 &
 \bigg(\int_\mathbb X \bigg(\int_{B(a,|x|_a )} f(y)dy \bigg)^q u(x)dx \bigg)^{1/q}\nonumber
 \\
 &\le (r/p)^{1/r}  \bigg(\int_\mathbb X  \bigg( \int_{\mathbb X \backslash{B(a,|x|_a )}}u(y)dy \bigg)^{r/q} \bigg(\int_{\mathbb X \backslash{B(a,|x|_a )}} V^{-p}(y)  w(y)dy\bigg)^{-r/q} V^{-p}(x)  w(x)dx \bigg)^{1/r}\nonumber
 \\
  & \quad\times  \bigg(\int_\mathbb X  \bigg(\int_{B(a,|x|_a )} f(y)dy \bigg)^p V^{-p}(x) w(x)dx \bigg)^{1/p}\nonumber
 \\
 & = (r/p)^{1/r} \bigg(\int_0^\infty \bigg( \int_\rho^\infty U_1(t)dt \bigg)^{r/q} \bigg( \int_\rho^\infty \tilde V^{-p}(t) V_1(t)dt \bigg)^{-r/q} \tilde V^{-p}(\rho) V_1(\rho)d\rho\bigg)^{1/r}\nonumber
 \\
 &\quad\times  \bigg(\int_\mathbb X  \bigg(\int_{B(a,|x|_a )} f(y)dy \bigg)^p V^{-p}(x) w(x)dx \bigg)^{1/p}.\nonumber
  \end{align}
Using $ \int_s^\infty \tilde V^{-p}(t) V_1(t)dt = (p'/p) \tilde V^{1-p}(s) $  in the first factor and applying Proposition \ref{proposition2} to the second factor, we reach the inequality
{\small \begin{align}
 &
  \bigg(\int_\mathbb X \bigg(\int_{B(a,|x|_a )} f(y)dy \bigg)^q u(x)dx \bigg)^{1/q}\nonumber
  \\
  & \le (r/p)^{1/r} (p/p')^{1/q} p' \bigg(\int_0^\infty \bigg( \int_\rho^\infty U_1(t) dt\bigg)^{r/q}  \tilde{V}(\rho)^{(p-1)r/q} \tilde V^{-p}(\rho)V_1(\rho)d\rho \bigg)^{1/r} \bigg(\int_\mathbb X f^p(x) v(x)dx \bigg)^{1/p}\nonumber
  \\
  & = (r/p)^{1/r} (p/p')^{1/q} p' \bigg(\int_0^\infty \bigg( \int_\rho^\infty U_1(t) dt \bigg)^{r/q} \bigg(\int_0^\rho V_1(t) dt \bigg)^{r/q'} V_1(\rho)d\rho \bigg)^{1/r} \bigg(\int_\mathbb X f^p(x) v(x)dx \bigg)^{1/p}\nonumber
    \\
    & = (r/p)^{1/r} (p/p')^{1/q} p' 
    \mathcal A_1  \bigg(\int_\mathbb X f^p(x) v(x)dx \bigg)^{1/p}\nonumber
    \\
  & = (r/p)^{1/r} (p/p')^{1/q} p' (p'/q)^{1/r} \mathcal A_2  \bigg(\int_\mathbb X f^p(x) v(x)dx \bigg)^{1/p}\nonumber
  \\
  & =(r/q)^{1/r} (p')^{1/p'} p^{1/p}  \mathcal A_2  \bigg(\int_\mathbb X f^p(x) v(x)dx \bigg)^{1/p}.\nonumber
     \end{align}
    }
     In the fourth last equality, we used $\frac{r(p-1)}{q}-p=r\bigg(\frac{p-1}{q}-\frac{p}{r}\bigg)=r\bigg(\frac{p-1}{q}-\frac{p(p-q)}{pq}\bigg)=\frac{r}{q'}$ and in the second last equality, we used Lemma \ref{lemma1}.\\
     
  To establish sufficiency for general $w,$ we fix positive functions $u$ and $w$. If $w=0$ almost everywhere on some ball $B(a,\vert x \vert_a)$ then translating $u$, $w$ on the left will reduce the problem to the one in which this does not occur. (If $w=0$  almost everywhere on $ \mathbb X$, sufficiency holds trivially). We therefore assume that $0<\displaystyle \int_{B(a,|x|_a )} w(y)dy,$ for all $x \neq a $. For each $n>0,$ set $u_n=u \chi_{(B(a,n))}$ and $w_n=\min(w,n)+\chi_{( \mathbb X \backslash{B(a,n )})}$. Then $w_n$ clearly satisfies \eqref{EQ:Condn}, so from previous arguments we have
  \begin{align}  
  & \bigg(\int_\mathbb X \bigg(\int_{B(a,|x|_a )} f(y)dy \bigg)^q u_n(x)dx \bigg)^{1/q}\nonumber
  \\
  & \le c \bigg(\int_0^\infty \bigg( \int_\rho^\infty \tilde u_n(t)dt \bigg)^{r/p}  \bigg(\int_0^\rho \tilde w_n(t)dt \bigg)^{r/p'} \tilde u_n(\rho)d\rho \bigg)^{1/r} \bigg( \int_{\mathbb X}  f^p(y)  w_n^{1-p}(y)dy \bigg)^{1/p},\nonumber
    \end{align}
    for all $f\ge0$. Here $c=(r/q)^{1/r} (p')^{1/p'} p^{1/p}$. If we take $f = g \min(w,n)^{1/p'} \chi_{(B(a,n) )}$ and use the definitions of $u_n$ and $w_n,$ the inequality becomes 
\begin{align}
&\bigg(\int_{B(a,n)}   \bigg(\int_{B(a,\vert x \vert_a)}  {g}(y) \min( w,n)^{1/p'}dy \bigg)^q  u (x)dx \bigg)^{1/q}\le c \bigg(\int_0^n \bigg(\int_\rho^n \tilde u(t) dt\bigg)^{r/p}\nonumber
\\
& \quad\times  \bigg(\int_0^\rho \min(\tilde w,n)\bigg)^{r/p'} \tilde u(\rho)d\rho \bigg)^{1/r} \bigg(\int_ {\mathbb X}  g^p(y)dy \bigg)^{1/p},\nonumber 
\end{align}   
for all non-negative $g$. We let $n \rightarrow \infty$, apply the Monotone Convergence Theorem and substitute $fw^{-1/p'}$ for $g$ to get the desired inequality and complete the proof.
\end{proof}

  \section{Applications and examples}
 
 In this section we present several examples of applications of our results to characterise the weights $u$ and $v$ in several settings: homogeneous groups, hyperbolic spaces, and more general Cartan-Hadamard manifolds.
  
 \subsection{Homogeneous groups}  Let $\X=\mathbb G$ be a homogeneous group in the sense of Folland and Stein \cite{FS-Hardy}, see also an up-to-date exposition in the open access books \cite{FR} and \cite{RS-book}.
Here condition \eqref{EQ:polar} is always satisfied with function $ \lambda (r,\omega)= {r}^{Q-1}$, with $Q$ being the homogeneous dimension of the group. 


Without loss of generality, let us fix $a=0$ to be the identity element of the group $\mathbb G$. To simplify the notation further, we denote $\vert x \vert_a$ by $\vert x \vert$. We note that this is consistent with the notation for the quasi-norm $|\cdot|$ on a homogeneous group $\mathbb G$.

Let us consider an example  of the power weights
$$u(x) =
\left\{
	\begin{array}{ll}
		{\vert x \vert}^{\alpha_1}  & \mbox{if }  \vert x \vert < 1, \\
		{\vert x \vert}^{\alpha_2} & \mbox{if }  \vert x \vert \geq 1,
	\end{array}
\right.
\qquad v(x)= {\vert x \vert }^\beta.$$
Then by Theorem \ref{Theorem1} the inequality 
$$
\bigg(\int_\mathbb G \bigg( \int_{B(a,{\vert x \vert}_a)} \vert f(y) \vert dy \bigg)^q u(x) dx\bigg)^\frac{1}{q} \le C\bigg( \int_\mathbb G {\vert f(y) \vert}^p v(x) dx \bigg)^{\frac{1}{p}}
$$
 holds for $0<q<p,$ $1<p <\infty$, if and only if 
 \begin{align}
 &
 \mathcal A_2^r =\bigg(\int_\mathbb G \bigg(\int_{\mathbb G\backslash{B(a,|x|_a )}} u(y) dy\bigg)^{r/p}  \bigg( \int_{B(a,|x|_a  )}v^{1-p'}(y)dy\bigg)^{r/p'}u(x)dx\bigg)\nonumber
 \\
 & \approx \int_0^1 \bigg(\int_t^1 \rho^{\alpha_1+Q-1} d\rho+ \int_1^\infty \rho^{\alpha_2+Q-1} d\rho \bigg)^{r/p} \bigg(\int_0^t \rho^{\beta(1-p')+Q-1} d\rho \bigg)^{r/p'} t^{\alpha_1+Q-1}dt\nonumber 
 \\
& + \int_1^\infty \bigg(\int_t^\infty \rho^{\alpha_2+Q-1} d\rho \bigg)^{r/p} \bigg(\int_0^t \rho^{\beta(1-p')+Q-1} d\rho \bigg)^{r/p'} t^{\alpha_2+Q-1}dt <\infty. \nonumber
 \end{align}

Let us consider
\begin{align}
&
\int_0^1 \bigg(\int_t^1 \rho^{\alpha_1+Q-1} d\rho+ \int_1^\infty \rho^{\alpha_2+Q-1} d\rho \bigg)^{r/p} \bigg(\int_0^t \rho^{\beta(1-p')+Q-1} d\rho \bigg)^{r/p'} t^{\alpha_1+Q-1}dt\nonumber
\\
&=\int_0^1 \bigg(1/(\alpha_1+Q)- {t}^{\alpha_1+Q}/(\alpha_1+Q)-1/(\alpha_2+Q) \bigg)^{r/p} \bigg(t^{(\beta(1-p')+Q)}/(\beta(1-p')+Q) \bigg)^{r/p'}\nonumber
\\
& \quad\times t^{\alpha_1+Q-1}dt,\nonumber
\end{align}
which is finite for 
$$\alpha_2+Q<0, \beta(1-p')+Q>0, (\alpha_1+Q){r/p}+(\beta(1-p')+Q){r/p'}+\alpha_1+Q>0,$$
which means  
$$\alpha_2+Q<0, \beta(1-p')+Q>0, (\alpha_1+Q){r/q}+(\beta(1-p')+Q){r/p'}>0,$$ 
since we have  $r/p+1=r(1/p+1/r)=r(1/p+1/q-1/p)=r/q$.
\\

Now, consider the other part
\begin{align}
&
\int_1^\infty \bigg(\int_t^\infty \rho^{\alpha_2+Q-1} d\rho \bigg)^{r/p} \bigg(\int_0^t \rho^{\beta(1-p')+Q-1} d\rho \bigg)^{r/p'} t^{\alpha_2+Q-1}dt\nonumber
\\
&= \int_1^\infty  \{(-t)^{(\alpha_2+Q)}/(\alpha_2+Q)\}^{r/p} \{ t^{(\beta(1-p')+Q)}/(\beta(1-p')+Q) \}^{r/p'} t^{\alpha_2+Q-1} dt, \nonumber
\end{align}
which is finite for 
$$\alpha_2+Q<0, \beta(1-p')+Q>0, (\alpha_2+Q){r/p}+(\beta(1-p')+Q){r/p'}+\alpha_2+Q<0,$$
or for 
$$\alpha_2+Q<0,\beta(1-p')+Q>0, (\alpha_2+Q){r/q}+(\beta(1-p')+Q){r/p'}<0.$$

Summarising that we get

  \begin{cor}\label{COR:hom}
Let $\mathbb G$ be a homogeneous group of homogeneous dimension $Q$, an we equip it with a quasi-norm $|\cdot|$. Let  $0<q<p,$ $1<p <\infty,$ $1/r=1/q-1/p,$ and let $\alpha_1, \alpha_2, \beta \in \mathbb R$. Assume that $\alpha_1+Q\not=0$.
Let 
\begin{equation}\label{EQ:uv}
u(x) =
\left\{
	\begin{array}{ll}
		{\vert x \vert}^{\alpha_1}  & \mbox{if } \vert x \vert < 1, \\
		{\vert x \vert}^{\alpha_2} & \mbox{if }  \vert x \vert \geq 1,
	\end{array}
\right.
\qquad
 v(x)= {\vert x \vert }^\beta.
 \end{equation}
 Then the inequality
\begin{equation}\label{EQ:Hardy1hg}
\bigg(\int_\mathbb G\bigg(\int_{B(a,\vert x \vert_a)}\vert f(y) \vert dy\bigg)^q u(x) dx\bigg)^\frac{1}{q}\le C\bigg\{\int_{\mathbb G} {\vert f(x) \vert}^p v(x) dx\bigg\}^{\frac1p}
\end{equation}
holds for all measurable functions $f:\mathbb G\to{\mathbb C}$  if and only if 
the parameters satisfy the following conditions:
 $\alpha_2+Q<0$, $\beta(1-p')+Q>0$, $(\alpha_1+Q){r/q}+(\beta(1-p')+Q){r/p'}>0,$ $(\alpha_2+Q){r/q}+(\beta(1-p')+Q){r/p'}<0.$
 \end{cor} 
 
It is interesting to note that in view of the last two conditions, it is not possible to have Hardy inequality \eqref{EQ:Hardy1hg} with weights $u$ and $v$ in \eqref{EQ:uv} with $\alpha_1=\alpha_2$. This is why we consider different powers $\alpha_1, \alpha_2$ in this example. This is different from the case $p\leq q$ which was considered as an application in \cite{RV-18}.
 
The case $\alpha_1+Q=0$ can be treated in a similar way.

\subsection{Hyperbolic spaces} 

 Let $\mathbb H^n$ denote the hyperbolic space of dimension $n$. In this case  condition \eqref{EQ:polar} is always satisfied with $\lambda(r,\omega)=(\sinh {r})^{n-1}$. Let $a\in \mathbb H^n$, and
 let us fix the weights 
 $$u(x) =
 \left\{
	\begin{array}{ll}
		{(\sinh {\vert x \vert_a})}^{\alpha_1}  & \mbox{if }  \vert x \vert < 1, \\
		{(\sinh {\vert x \vert_a})}^{\alpha_2} & \mbox{if }  \vert x \vert \geq 1,
	\end{array}
\right.,\quad v(x)= (\sinh {\vert x \vert_a})^\beta.$$  
 We note that $\mathcal A_2$ is equivalent to
 \begin{align}
& \mathcal A_2^r \approx \int_0^1  \bigg(\int_{t}^1 (\sinh \rho)^{\alpha_1+n-1}d\rho + \int_{1}^\infty (\sinh \rho)^{\alpha_2+n-1}d\rho \bigg)^{r/p} \bigg( \int_0^{t} (\sinh \rho)^{\beta(1-p')+n-1} d\rho \bigg)^{r/p'}\nonumber
\\
& \quad\times( \sinh {t})^{\alpha_1+n-1} dt + \int_1^\infty\bigg( \int_{t}^\infty (\sinh \rho)^{\alpha_2+n-1}d\rho \bigg)^{r/p}  \bigg( \int_0^{t} (\sinh \rho)^{\beta(1-p')+n-1} d\rho \bigg)^{r/p'}  \nonumber
\\
 &\quad\times( \sinh {t})^{\alpha_2+n-1} dt. \nonumber
\end{align}

In the first part, for ${\alpha_2+n-1}<0$ and $ {\beta(1-p')+n}>0,$
 \begin{align}
&\int_0^1  \bigg(\int_{t}^1 (\sinh \rho)^{\alpha_1+n-1}d\rho + \int_{1}^\infty (\sinh \rho)^{\alpha_2+n-1}d\rho \bigg)^{r/p} \bigg( \int_0^{t} (\sinh \rho)^{\beta(1-p')+n-1} d\rho \bigg)^{r/p'}\nonumber
\\
& \quad\times( \sinh {t})^{\alpha_1+n-1} dt\nonumber
\\
&\approx \int_0^1  \bigg(\int_{t}^1 ( \rho)^{\alpha_1+n-1}d\rho + \int_{1}^\infty (\exp \rho)^{\alpha_2+n-1}d\rho \bigg)^{r/p}\bigg( \int_0^{t} ( \rho)^{\beta(1-p')+n-1} d\rho \bigg)^{r/p'}\nonumber
\\
& \quad\times( {t})^{\alpha_1+n-1} dt\nonumber
\\
&=\int_0^1 \bigg(1/(\alpha_1+n)- {t}^{\alpha_1+n}/(\alpha_1+n)-(\exp 1)^{\alpha_2+n-1}/({\alpha_2+n-1}) \bigg)^{r/p}\nonumber
\\
& \quad\times \bigg( t^{\beta(1-p')+n}/({\beta(1-p')+n}) \bigg)^{r/p'} t^{\alpha_1+n-1} dt,\nonumber
\end{align}
which is finite for, 

(a) $\alpha_1+n \geq 0$, $({\beta(1-p')+n})r/p' +{\alpha_1+n} >0,  $
\\

(b) $\alpha_1+n < 0$, $(\alpha_1+n){r/p}+{(\beta(1-p')+n})r/p' +{\alpha_1+n} >0.  $
\\

However, we can note that in (a), if $\alpha_1+n \geq 0$, then the second condition is automatically satisfied under our assumption $ {\beta(1-p')+n}>0.$
 
In the second part, for  ${\alpha_2+n-1}<0,$
 \begin{align}
 &\int_1^\infty \bigg( \int_{t}^\infty (\sinh \rho)^{\alpha_2+n-1}d\rho \bigg)^{r/p}  \bigg( \int_0^{t} (\sinh \rho)^{\beta(1-p')+n-1} d\rho \bigg)^{r/p'}
( \sinh {t})^{\alpha_2+n-1} dt \nonumber
\\
& \approx  \int_1^\infty \bigg( \int_{t}^\infty (\exp \rho)^{\alpha_2+n-1}d\rho \bigg)^{r/p} \bigg( \int_0^{t} (\exp \rho)^{\beta(1-p')+n-1} d\rho \bigg)^{r/p'}  (\exp {t})^{\alpha_2+n-1} dt \nonumber
\\
&= \int_1^\infty  \bigg( -(\exp t)^{\alpha_2+n-1}/(\alpha_2+n-1) \bigg)^{r/p}\bigg( (\exp t)^{\beta(1-p')+n-1}/(\beta(1-p')+n-1) \bigg)^{r/p'}  \nonumber
\\
& \quad\times ( (\exp {t})^{\alpha_2+n-1} dt, \nonumber
 \end{align}
 which is finite for
 $$
 ({\alpha_2+n-1})r/p+ (\beta(1-p')+n-1)r/p' + \alpha_2+n-1<0,
 $$
which is the same as
$$
 ({\alpha_2+n-1})r/q+ (\beta(1-p')+n-1)r/p' <0.
$$
\begin{cor}\label{COR:hyp}
Let $\mathbb H^n$ be the hyperbolic space, $a\in \mathbb H^n$, and let $|x|_a$ denote the hyperbolic distance from $x$ to $a$. Let $0<q<p,$ $1<p <\infty,$ $1/r=1/q-1/p,$ and let $\alpha_1, \alpha_2, \beta\in\mathbb R$. Assume that $\alpha_1+n\not=0$.
Let 
$$u(x) =
 \left\{
	\begin{array}{ll}
		{(\sinh {\vert x \vert_a})}^{\alpha_1}  & \mbox{if }  \vert x \vert < 1, \\
		{(\sinh {\vert x \vert_a})}^{\alpha_2} & \mbox{if }  \vert x \vert \geq 1,
	\end{array}
\right.
\qquad
 v(x)= (\sinh {\vert x \vert_a})^\beta.$$
 Then the inequality
\begin{equation*}\label{EQ:Hardy1hyp}
\bigg(\int_{\mathbb H^n}\bigg(\int_{B(a,\vert x \vert_a)}\vert f(y) \vert dy\bigg)^q u(x) dx\bigg)^\frac{1}{q}\le C\bigg\{\int_{\mathbb H^n} {\vert f(x) \vert}^p v(x) dx\bigg\}^{\frac1p}
\end{equation*}
holds for all measurable functions $f:\mathbb H^n\to{\mathbb C}$  if and only if 
the parameters satisfy the following conditions:
${\alpha_2+n-1}<0,$ $ {\beta(1-p')+n}>0,$ $ (\alpha_1+n){r/q}+(\beta(1-p')+n){r/p'}>0,$ $({\alpha_2+n-1})r/q+ (\beta(1-p')+n-1)r/p' <0.$ 
\end{cor}


\subsection{Cartan-Hadamard manifolds} \label{SEC:CH}
Let $(M,g)$ be a Cartan-Hadamard manifold. This means that $M$ is 
a complete and simply connected Riemannian manifold with non-positive sectional curvature, that is, the sectional curvature of $M$ satisfies $K_M\le 0$ along each (plane) section at each point of $M$. Then condition \eqref{EQ:polar} is automatically satisfied by taking $\lambda(\rho,\omega)= J(\rho,\omega) \rho^{n-1}$, where $J(\rho,\omega)$ is the density function on $M$, see e.g. \cite{DV1}, Helgason \cite{DV3}.  

Let us fix a point $a\in M$ and denote by 
$\rho(x)=d(x,a)$ the geodesic distance from $x$ to $a$ on $M$. The exponential map ${\rm exp}_a :T_a M \to  M$ is a diffeomorphism, see e.g. Helgason \cite{DV3}.  Let us assume that the sectional curvature $K_M$ is constant, in which case it is known that the function $J(t,\omega)$ depends only on $t$. More precisely, let us denote $K_M=-b$ for $b\ge0$. Then we have
$J(t,\omega)= 1$ if $b=0$, and $J(t,\omega)=(\frac{\sinh \sqrt{b}t}{\sqrt{b}t})^{n-1}$ for $b>0$,
see e.g. \cite {DV5}.
 In the case $b=0$, then let us take the weights 
 $$u(x) =
 \left\{
	\begin{array}{ll}
		{(\sinh {\vert x \vert_a})}^{\alpha_1}  & \mbox{if }  \vert x \vert < 1, \\
		{(\sinh {\vert x \vert_a})}^{\alpha_2} & \mbox{if }  \vert x \vert \geq 1,
	\end{array}
\right. \qquad
v(x)= (\sinh {\vert x \vert_a})^\beta,$$ 
then   inequality \eqref{EQ:Hardy1} holds for $0<q<p$, $1<p <\infty,$ $1/r=1/q-1/p,$ if and only if\\
$ \mathcal A_2 \approx \bigg(\displaystyle\int_0^1  \bigg( \int_t^1 \rho^{\alpha_1+n-1} d\rho +\int_1^\infty \rho^{\alpha_2+n-1} d\rho  \bigg)^{r/p} \bigg( \int_0^t \rho^{\beta(1-p')+n-1} d\rho \bigg)^{r/p'}  t^{\alpha_1+n-1}dt + \displaystyle+ \int_1^\infty  \bigg( \int_t^\infty \rho^{\alpha_2+n-1} d\rho   \bigg)^{r/p} \bigg( \int_0^t \rho^{\beta(1-p')+n-1} d\rho \bigg)^{r/p'}  t^{\alpha_2+n-1}dt\bigg)^{1/r}<\infty,\nonumber
$\\
which is finite if and only if conditions of Corollary \ref{COR:hom} hold with $Q=n$ (which is natural since the curvature is zero).
  \\
When $b>0$, let 
$$u(x) =
 \left\{
	\begin{array}{ll}
		{(\sinh \sqrt{b}{\vert x \vert_a})}^{\alpha_1}  & \mbox{if }  \vert x \vert < 1, \\
		{(\sinh \sqrt{b}{\vert x \vert_a})}^{\alpha_2} & \mbox{if }  \vert x \vert \geq 1,
	\end{array}
\right.
\qquad
v(x)=(\sinh \sqrt b{\vert x \vert_a})^\beta.$$
Then   inequality \eqref{EQ:Hardy1} holds for  $0<q<p$, $1<p <\infty,$ $1/r=1/q-1/p$, if and only if 
  $\mathcal A_2$ is finite. We have
 {\small 
 \begin{align}
& \mathcal A_2 \approx \bigg(\int_0^1  \bigg(\int_t^1 (\sinh \sqrt b \rho)^{\alpha_1} {(\frac{\sinh \sqrt b \rho}{{\sqrt b} \rho})}^{n-1} {\rho}^{n-1}d\rho + \int_1^\infty (\sinh \sqrt b \rho)^{\alpha_2} {(\frac{\sinh \sqrt b \rho}{{\sqrt b} \rho})}^{n-1} {\rho}^{n-1}d\rho \bigg)^{r/p}\nonumber
\\
&\quad\times  \bigg(\displaystyle \int_0^{t} (\sinh \sqrt b \rho)^{\beta(1-p')} {(\frac{\sinh \sqrt b \rho}{{\sqrt b} \rho})}^{n-1} {\rho}^{n-1}d\rho \bigg)^{r/p'}\nonumber
\\
& \quad\times ( \sinh \sqrt b {t})^{\alpha_1} {(\frac{\sinh \sqrt b t}{{\sqrt b} t})}^{n-1} {t}^{n-1}dt +\int_1^\infty  \bigg(\int_t^\infty (\sinh \sqrt b \rho)^{\alpha_2} {(\frac{\sinh \sqrt b \rho}{{\sqrt b} \rho})}^{n-1} {\rho}^{n-1}d\rho \bigg)^{r/p}\nonumber
\\
&\quad\times\bigg(\displaystyle \int_0^{t} (\sinh \sqrt b \rho)^{\beta(1-p')} {(\frac{\sinh \sqrt b \rho}{{\sqrt b} \rho})}^{n-1} {\rho}^{n-1}d\rho \bigg)^{r/p'}\nonumber
\\
& \quad\times ( \sinh \sqrt b {t})^{\alpha_2} {(\frac{\sinh \sqrt b t}{{\sqrt b} t})}^{n-1} {t}^{n-1}dt  \bigg)^{1/r}\nonumber
\\
&  \approx \bigg(\int_0^1  \bigg(\int_{t}^1 (\sinh \sqrt b \rho)^{\alpha_1+n-1} d\rho+\int_1^\infty (\sinh \sqrt b \rho)^{\alpha_2+n-1} d\rho \bigg)^{r/p}\nonumber
\\
&\quad\times  \bigg(\displaystyle \int_0^{t} (\sinh \sqrt b \rho)^{\beta(1-p')+n-1} d\rho \bigg)^{r/p'}( \sinh \sqrt b {t})^{\alpha_1+n-1} dt + \int_1^\infty  \bigg(\int_{t}^\infty (\sinh \sqrt b \rho)^{\alpha_2+n-1} d\rho \bigg)^{r/p
}\nonumber
\\
&\quad\times  \bigg(\displaystyle \int_0^{t} (\sinh \sqrt b \rho)^{\beta(1-p')+n-1} d\rho \bigg)^{r/p'}( \sinh \sqrt b {t})^{\alpha_2+n-1} dt \bigg)^{1/r},\nonumber
\end{align} }%
which has the same conditions for finiteness as the case of the hyperbolic space in Corollary \ref{COR:hyp}  (which is also natural since it is the negative constant curvature case).

 \end{document}